\documentclass[11pt,a4paper,leqno]{amsart}
\usepackage{cite,mathrsfs}
\usepackage{amsmath,amssymb,amsthm}
\usepackage{graphicx}
\usepackage{mathrsfs}
\usepackage{enumerate}

\usepackage{verbatim}
\usepackage{amssymb}
\usepackage{amsbsy}
\usepackage{amscd}
\usepackage{amsmath}
\usepackage{amsthm}
\usepackage{mathrsfs}
\usepackage{eucal}

\newtheorem{theorem}{Theorem}\numberwithin{theorem}{section}
\newtheorem{prop}[theorem]{Proposition}
\newtheorem{corollary}[theorem]{Corollary}
\newtheorem{lemma}[theorem]{Lemma}
\newtheorem{problem}[theorem]{Problem}

\newtheorem{defn}{Definition}%\numberwithin{defn}{section}

\theoremstyle{remark}
\newtheorem{remark}[theorem]{Remark}
\newtheorem{example}[theorem]{Example}

\numberwithin{equation}{section}

\def\bb{B}

\def\cc{C}

\def\kk{k_1,\ldots,k_r}
\def\km{-k_1,\ldots,-k_r}

\def\Li{{\rm Li}}
\def\T{T_0}
\def\TT{T}

\long\def\comment#1{}

\DeclareFontEncoding{OT2}{}{}
\DeclareFontSubstitution{OT2}{cmr}{m}{n}
\DeclareSymbolFont{cyss}{OT2}{wncyss}{m}{n}
\DeclareMathSymbol{\sh}{\mathbin}{cyss}{`x}

\allowdisplaybreaks

\title[Zeta functions connecting MZVs and poly-Bernoulli numbers]
{Zeta functions connecting multiple zeta values and poly-Bernoulli numbers}

\author[M. Kaneko]{Masanobu Kaneko}
\author[H. Tsumura]{Hirofumi Tsumura}

\address{M.\,Kaneko: Faculty of Mathematics, Kyushu University, Motooka 744, Nishi-ku, Fukuoka 819-0395, Japan}
\email{mkaneko@math.kyushu-u.ac.jp}

\address{{H.\,Tsumura:} Department of Mathematical Sciences, Tokyo Metropolitan University, 1-1, Minami-Ohsawa, Hachioji, Tokyo 192-0397, Japan}
\email{tsumura@tmu.ac.jp}



\begin{document}
\baselineskip 14pt

%\dedicatory{\footnotesize Dedicated to Professor Kohji Matsumoto, with admiration}

\begin{abstract}
We first review our previous works of Arakawa and the authors on two, closely related single-variable zeta functions. 
Their special values at positive and negative integer arguments are respectively multiple zeta values and 
poly-Bernoulli numbers. We then introduce, as a generalization of Sasaki's work, level 2 analogue
of one of the two zeta functions and prove results analogous to those by Arakawa and the first named
author.
\end{abstract}

\maketitle

%\begin{center}
%\end{center}

%\baselineskip 16pt
%%%%%%%%%%%%%%%%%%%%%%%%%%%%%%%%%%%%%%%%%%%%                                   
\section{Introduction} \label{sec-1}
%%%%%%%%%%%%%%%%%%%%%%%%%%%%%%%%%%%%%%%%%%%%  

In this (half expository) paper, we discuss some properties of two single-variable functions $\xi_k(s)$ 
and $\eta_k(s)$, which are closely related with each other, and their generalizations. 
We are interested in these functions because multiple zeta values and poly-Bernoulli numbers 
appear as special values, respectively at positive and negative integer arguments.

The multiple zeta value (MZV) and its variant multiple zeta-star value (MZSV), a vast amount of researches 
on which from various points of view has been carried out in recent years, are defined by 
\begin{align*}
& \zeta(k_1,\ldots,k_r)=\sum_{1 \le m_1< \cdots< m_r}\frac{1}{m_1^{k_1}\cdots m_r^{k_r}}\\
\intertext{and}
& \zeta^\star(k_1,\ldots,k_r)=\sum_{1\leq m_1\leq \cdots\leq m_r}\frac{1}{m_1^{k_1}\cdots m_r^{k_r}}
\end{align*}
for $k_1,\ldots,k_r\in \mathbb{Z}_{\geq 1}$ with $k_r>1$ (for convergence), respectively.  MZVs appear as special values
of $\xi_k(s)$ and MZSV as those of $\eta_k(s)$  (Theorem \ref{T-3-7}).  

Poly-Bernoulli numbers, having also two versions $B_n^{(k)}$ and $C_n^{(k)}$,  were defined by 
the first named author in \cite{Kaneko1997} and 
in Arakawa-Kaneko \cite{AK1999} by using generating series:  For an integer $k\in \mathbb{Z}$, 
define sequences of rational numbers 
$\{\bb_n^{(k)}\}$ and $\{\cc_n^{(k)}\}$ by 
\begin{align}
&\frac{{\rm Li}_{k}(1-e^{-t})}{1-e^{-t}}=\sum_{n=0}^\infty \bb_n^{(k)}\frac{t^n}{n!} \label{1-1}\\
\intertext{and}
&\frac{{\rm Li}_{k}(1-e^{-t})}{e^t-1}=\sum_{n=0}^\infty \cc_n^{(k)}\frac{t^n}{n!},  \label{1-2}
\end{align}
where ${\rm Li}_{k}(z)$ is the polylogarithm function (or rational function when $k\le0$) defined by
\begin{equation}
{\rm Li}_{k}(z)=\sum_{m=1}^\infty \frac{z^m}{m^k}\quad (|z|<1). \label{1-3}
\end{equation}
Since ${\rm Li}_1(z)=-\log(1-z)$, the generating functions on the left-hand sides respectively become 
\[ \frac{te^t}{e^t-1} \quad \text{and} \quad \frac{t}{e^t-1} \] when $k=1$, and hence 
$B_n^{(1)}$ and $C_n^{(1)}$
are usual Bernoulli numbers, the only difference being $B_1^{(1)}=1/2$ and $C_1^{(1)}=-1/2$. 
When $k\ne1$, $B_n^{(k)}$'s and $C_n^{(k)}$'s are totally different numbers.
We mention in passing that $B_n^{(-k)}$ ($n,k\ge0$) coincides with the number of acyclic orientations of
the complete bipartite graph $K_{n,k}$ (see \cite{Cameron}), and is also equal to the number of 
`lonesum' matrices of size $n\times k$ (see \cite{Brew2008}). 

In \cite{KT} and \cite{AK1999}, we showed that poly-Bernoulli numbers $B_n^{(k)}$ and $C_n^{(k)}$ 
appear as special values at nonpositive integers 
of $\eta_k(s)$ and $\xi_k(s)$ respectively. Multi-indexed version of these results were established
in  \cite{KT} and will be reviewed in \S\ref{sec-2}  (\eqref{xi-value-1} and \eqref{3-2}). 

In \S\ref{sec-3}, we give formulas obtained in \cite{KT} relating $\xi$ and $\eta$ (Proposition \ref{etaxi}) and also  
an expression of $\xi$ in terms of multiple zeta functions (Theorems \ref{xibyzeta}). 

In \S\ref{sec-4}, 
we focus on the duality properties of $B_n^{(k)}$ and $C_n^{(k)}$, namely 
\begin{align}
& B_n^{(-k)}=B_{k}^{(-n)}, \label{1-4}\\
& C_n^{(-k-1)}=C_{k}^{(-n-1)} \label{1-5}
\end{align}
for $k,n\in \mathbb{Z}_{\geq 0}$ 
(see \cite[Theorems\ 1\ and\ 2]{Kaneko1997} and \cite[\S\,2]{Kaneko-Mem}).
We can interpret \eqref{1-4} and \eqref{1-5} as the identities 
\[ \eta_{-k}(-n)=\eta_{-n}(-k)\quad\text{and}\quad \widetilde{\xi}_{-k-1}(-n)=\widetilde{\xi}_{-n-1}(-k)\] 
for $k,n\in \mathbb{Z}_{\geq 0}$, respectively, where $\widetilde{\xi}_{-k}(s)$ is another type of 
function interpolating $C_n^{(k)}$ (see \eqref{4-6-2}). 
These relations even hold if we extend $k$ and $n$ to complex variables, 
as shown by Yamamoto \cite{Yamamoto} and Komori-Tsumura \cite{Ko-Tsu} (see \eqref{eta-dual-3} and 
\eqref{tildexi-dual}).

In \S\ref{sec-6}, we generalize Sasaki's zeta function (see \cite{Sasaki2012}) from the viewpoint that 
it gives a level $2$-version of 
$\xi(k_1,\ldots,k_r;s)$. Our previous methods work well in this case and we obtain several formulas 
related to multiple zeta values of level $2$. 
This section is substantially new.

%%%%%%%%%%%%%%%%%%%%%%%%%%%%%%%%%%%%%%%%%%
\section{Multi-poly-Bernoulli numbers and related zeta functions}\label{sec-2}
%%%%%%%%%%%%%%%%%%%%%%%%%%%%%%%%%%%%%%%%%%

Imatomi, Takeda, and the first named author \cite{IKT2014} introduced multi-index generalizations of poly-Bernoulli numbers  
(``multi-poly-Bernoulli numbers") as follows.

\begin{defn}\label{def-MPBer}
For $\kk\in \mathbb{Z}$, define two types of multiple poly-Bernoulli numbers by
\begin{align}
&\frac{{\rm Li}_{k_1,\ldots,k_r}(1-e^{-t})}{1-e^{-t}}=\sum_{n=0}^\infty \bb_n^{(\kk)}\frac{t^n}{n!}  \label{1-6}\\
\intertext{and}
&\frac{{\rm Li}_{k_1,\ldots,k_r}(1-e^{-t})}{e^{t}-1}=\sum_{n=0}^\infty \cc_n^{(\kk)}\frac{t^n}{n!},  \label{1-6-2}
\end{align}
where 
\begin{equation}
{\rm Li}_{k_1,\ldots,k_r}(z)=\sum_{1\leq m_1<\cdots<m_r}\frac{z^{m_r}}{m_1^{k_1}m_2^{k_2}\cdots m_r^{k_r}}  \label{1-7}
\end{equation}
is the multiple polylogarithm. 
\end{defn}

\begin{remark} 
In \cite{IKT2014}, 
the following relation between $\cc_{p-2}^{(\kk)}$ and the `finite multiple zeta value' was proved:
\begin{equation}
\sum_{1\leq m_1<\cdots <m_r<p}\frac{1}{m_1^{k_1}\cdots m_r^{k_r}} \equiv -C_{p-2}^{(k_1,\ldots,k_{r-1},k_r-1)} \mod p \label{1-8}
\end{equation}
for any prime number $p$.
\end{remark}

In connection with these numbers, we consider the following two types of zeta functions. The first one,
$\xi(\kk;s)$, was defined in \cite{AK1999} as follows.

\begin{defn}\label{def-xi}
For $r\in \mathbb{Z}_{\geq 1}$, $k_1,\ldots,k_r\in \mathbb{Z}_{\geq 1}$ and $\Re s>0$, 
\begin{equation}
\xi(\kk;s)=\frac{1}{\Gamma(s)}\int_0^\infty {t^{s-1}}\frac{\Li_{\kk}(1-e^{-t})}{e^t-1}\,dt, \label{xidef}
\end{equation}
where $\Gamma(s)$ is the gamma function. 
In the case $r=1$, denote $\xi(k;s)$ by $\xi_k(s)$. Note that $\xi_1(s)=s\zeta(s+1)$. 
\end{defn}

This can be analytically continued to an entire function for $s\in\mathbb{C}$, and satisfies the following (see \cite[Remark 2.4]{AK1999}):
\begin{equation}
\xi(\kk;-m)=(-1)^m\cc_{m}^{(\kk)}\quad (m\in \mathbb{Z}_{\geq 0}) \label{xi-value-1}
\end{equation}
for $(\kk)\in \mathbb{Z}_{\geq 1}^r$. This can be regarded as a poly-analogue of the classical 
evaluation \[\xi_1(-m)=(-m)\zeta(1-m)=(-1)^mC_m.\]

The second, $\eta(\kk;s)$, is defined as follows (see \cite{KT}).

\begin{defn}\label{Def-Main-1} 
For $r\in \mathbb{Z}_{\geq 1}$, $k_1,\ldots,k_r\in \mathbb{Z}_{\geq 1}$ and $\Re s>1-r$,  
\begin{equation}
\eta(\kk;s)=\frac{1}{\Gamma(s)}\int_{0}^\infty t^{s-1}\frac{\Li_{\kk}(1-e^t)}{1-e^t}dt \label{etadef}
\end{equation}
for $s\in \mathbb{C}$ with ${\rm Re}(s)>1-r$, In the case $r=1$, denote $\eta(k;s)$ by $\eta_k(s)$. Note that $\eta_1(s)=s\zeta(s+1)$. 
\end{defn}

Similar to $\xi(\kk;s)$, we see that $\eta(\kk;s)$ can be analytically continued to an entire function for 
$s\in\mathbb{C}$, and satisfies the following (see \cite[Theorem 2.3]{KT}):
\begin{equation}
\eta(\kk;-m)=\bb_{m}^{(\kk)}\quad (m\in \mathbb{Z}_{\geq 0}) \label{3-2}
\end{equation}
for positive integers $k_1,\ldots,k_r\in \mathbb{Z}_{\geq 1}$. 
This can be regarded as a poly-analogue of \[\eta_1(-m)=(-m)\zeta(1-m)=B_m.\]

As for their values at positive integers, we can obtain explicit expressions in terms of 
multiple zeta/zeta-star values as follows. We prepare several notations. 
For an index set ${\bf k}=(\kk)\in\mathbb{Z}_{\ge1}^r$, 
put ${\bf k}_+=(k_1,\ldots,k_{r-1},k_r+1)$. 
The usual dual index of an admissible index ${\bf k}$ is denoted by ${\bf k}^*$. For ${\bf j}=(j_1,\ldots,j_r)\in\mathbb{Z}_{\ge0}^r$, 
we set $\vert{\bf j}\vert=j_1+\cdots+j_r$ and call it the weight of ${\bf j}$, 
and $d({\bf j})=r$, the depth of ${\bf j}$. For two such indices ${\bf k}$ and ${\bf j}$ of the same depth, 
we denote by ${\bf k}+{\bf j}$ the index obtained by the component-wise addition,  
${\bf k}+{\bf j}=(k_1+j_1, \ldots, k_r+j_r)$, and by $b({\bf k};{\bf j})$ the quantity given by
\[   b({\bf k};{\bf j}):=\prod_{i=1}^r \binom{k_i+j_i-1}{j_i}. \]

\begin{theorem}[\cite{KT}\ Theorem 2.5]\label{T-3-7}\ For any index set ${\bf k}=(\kk)\in\mathbb{Z}_{\ge1}^r$
and any $m\in \mathbb{Z}_{\geq 1}$, we have 
\begin{equation}\label{xivalue} 
\xi(\kk;m)=\sum_{\vert{\bf j}\vert=m-1,\,d({\bf j})=n} b(({\bf k}_+)^*;{\bf j})\,
\zeta(({\bf k}_+)^*+{\bf j})
\end{equation}
and
\begin{equation}\label{etavalue}
\begin{split}
\eta(\kk;m)=(-1)^{r-1} & \sum_{\vert{\bf j}\vert=m-1,\, d({\bf j})=n} b(({\bf k}_+)^*;{\bf j})\\
 & \qquad \times \zeta^\star(({\bf k}_+)^*+{\bf j}),
\end{split}
\end{equation}
where both sums run over all ${\bf j}\in\mathbb{Z}_{\ge0}^r$ of weight $m-1$ and depth 
$n:=d({\bf k}_+^*) \ (=\vert{\bf k}\vert+1-d({\bf k}))$.

In particular, we have
\[ \qquad \xi(\kk;1)=\zeta({\bf k}_+) \]
and
\[ \eta(\kk;1)=(-1)^{r-1}\zeta^\star(({\bf k}_+)^*).\]
\end{theorem}
Here we have used the duality $\zeta(({\bf k}_+)^*)=\zeta({\bf k}_+)$.

\begin{remark}
In \cite[Theorem 9 (i)]{AK1999}, we proved \eqref{xivalue} in the case when $(\kk)=(1,\ldots,1,k)$. 
The above formulas generalize this and give its $\eta$-version. In fact, these can be proved by the same method as in \cite{AK1999}, i.e., 
by considering the integral expressions 
\begin{align*}
\zeta(k_1,\ldots,k_r) \ &= \frac{1}{\prod_{j=1}^{r}\Gamma(k_j)}\int_0^\infty\!\!\cdots\int_0^\infty 
\frac{x_1^{k_1-1}\cdots x_r^{k_r-1}}{e^{x_1+\cdots+x_r}-1}\\
& \qquad \times  \frac1{e^{x_{2}+\cdots+x_{r}}-1}\cdots \frac1{e^{x_r}-1}dx_1\cdots dx_r, \\
\zeta^\star(k_1,\ldots,k_r) 
 &= \frac{1}{\prod_{j=1}^{r}\Gamma(k_j)}\int_0^\infty\!\!\cdots\int_0^\infty 
\frac{x_1^{k_1-1}\cdots x_r^{k_r-1}}{e^{x_1+\cdots+x_r}-1}\\
& \qquad\times \frac{e^{x_{2}+\cdots+x_r}}{e^{x_{2}+\cdots+x_{r}}-1} \cdots \frac{e^{x_r}}{e^{x_r}-1}dx_1\cdots dx_r
\end{align*}
for $k_1,\ldots,k_r\in \mathbb{Z}_{\geq 1}$ with $k_r\geq 2$.

We emphasize that the formulas \eqref{xivalue} and \eqref{etavalue} have remarkable similarity in that one obtains \eqref{etavalue} 
just by replacing multiple zeta values in \eqref{xivalue} with multiple zeta-star values. 
\end{remark}

Noting the duality $(k+1)^*=(\underbrace{1,\ldots,1}_{k-1},2)$, we can obtain the following two identities.
The former is a special case of \cite[Theorem 9 (i)]{AK1999} and the latter is \cite[Corollary 2.8]{KT}. 

\begin{corollary}  For $k,m\ge1$, we have
\begin{align}
&\xi_k(m)=\sum_{j_1,\ldots,j_{k-1}\ge1, j_k\ge2\atop j_1+\cdots+j_k=k+m}
(j_k-1)\zeta(j_1,\ldots,j_{k-1},j_k), \label{3-8}\\
& \eta_k(m)=\sum_{j_1,\ldots,j_{k-1}\ge1, j_k\ge2\atop j_1+\cdots+j_k=k+m}
(j_k-1)\zeta^\star(j_1,\ldots,j_{k-1},j_k). \label{3-8-2}
\end{align}
\end{corollary}

%%%%%%%%%%%%%%%%%%%%%%%%%%%%%%%%%%%%%%%%%%%%%%%%
\section{Relations among $\xi,\,\eta$ and multiple zeta functions}\label{sec-3}
%%%%%%%%%%%%%%%%%%%%%%%%%%%%%%%%%%%%%%%%%%%%%%%%

In this section, we give formulas describing relations among $\xi,\,\eta$ and multiple zeta functions by 
employing two types of connection formulas for the multiple polylogarithm. 

First we show that each of the functions $\eta$ and $\xi$ can be written 
as a linear combination of the other in exactly the same way, using the 
so-called Landen-type connection formula for the multiple polylogarithm $\Li_{\kk}(z)$. 

For two indices ${\bf k}$ and ${\bf k'}$ of 
the same weight, we say ${\bf k'}$ refines ${\bf k}$, denoted ${\bf k}\preceq {\bf k'}$,
if ${\bf k}$ is obtained from ${\bf k'}$ by replacing some commas by $+$'s. 
For example, 
\[ (3)=(1+1+1)\preceq(1,1,1),\  (2,3)=(2,2+1)\preceq(2,2,1)\]
etc. 
Using this notation, the Landen connection formula for the multiple polylogarithm is as follows.

\begin{lemma}[Okuda-Ueno\ \cite{OU}\ Proposition 9]\label{landen} For any index ${\bf k}$ of depth $r$, we have
\begin{equation}  
\Li_{\bf k}\biggl(\frac{z}{z-1}\biggr)=(-1)^r\sum_{{\bf k}\preceq{\bf k'}}\Li_{\bf k'}(z). \label{Landen-F}
\end{equation}
\end{lemma}

Using \eqref{Landen-F} for the case $z=1-e^{-t}$ (resp. $1-e^t$), namely $z/(z-1)=1-e^t$ (resp. $1-e^{-t}$), we can prove the following.

\begin{prop}[\cite{KT}\ Proposition 3.2]\label{etaxi} Let ${\bf k}$ be any index set and $r$ its depth.  We have the relations
\begin{equation} \eta({\bf k};s)= (-1)^{r-1}\sum_{{\bf k}\preceq{\bf k'}}\xi({\bf k'};s)\label{etabyxi}
\end{equation}
and
\begin{equation}\xi({\bf k};s)= (-1)^{r-1}\sum_{{\bf k}\preceq{\bf k'}}\eta({\bf k'};s).\label{xibyeta}
\end{equation}
\end{prop}
The reason of the symmetry is that the transformation $z\to z/(z-1)$ is involutive. 

Here we recall a certain formula between $\xi$ and the single-variable multiple zeta function
\begin{equation}
\zeta(k_1,\ldots,k_r;s)=\sum_{1\leq m_1<\cdots<m_r<m}\frac{1}{m_1^{k_1}\cdots m_r^{k_r}m^s}  \label{mzfct}
\end{equation}
defined for integers $k_1,\ldots,k_r$ as follows.

\begin{theorem}[\cite{AK1999}\ Theorem 8] \label{AK-Th8} 
For $r,k\in \mathbb{Z}_{\geq 1}$, 
\begin{align}
& \xi(\underbrace{1,\ldots,1}_{r-1},k;s)\label{AK-thm-8}\\
& =(-1)^{k-1}\sum_{a_1+\cdots+a_k=r \atop \forall a_j \geq 0}\binom{s+a_k-1}{a_k}\zeta(a_1+1,\ldots,a_{k-1}+1;a_k+s)\notag\\
& \qquad +\sum_{j=0}^{k-2}(-1)^j \zeta(\underbrace{1,\ldots,1}_{r-1},k-j)\zeta(\underbrace{1,\ldots,1}_{j};s).\notag
\end{align}
\end{theorem}

Concerning a generalization of this result, Arakawa and the first named author posed the following question. 

\begin{problem}[\cite{AK1999}\ \S8,\ Problem (i)]\label{Problem-1}
For a general index set $(k_1,\ldots,k_r)$, is the function $\xi(k_1,\ldots,k_r;s)$ also expressed by multiple zeta functions 
as in Theorem \ref{AK-Th8} stated above?
\end{problem}

An affirmative answer was given in \cite{KT}. To describe it, we consider an Euler-type connection formula 
for the multiple polylogarithm.

\begin{lemma}[\cite{KT}\ Lemma 3.5]\label{eulcon} Let ${\bf k}$ be any index.  Then we have
\begin{equation}
\Li_{\bf k}(1-z)=\sum_{{\bf k'},\,j\ge0}c_{\bf k}({\bf k'};j)
\Li_{{\scriptsize{\underbrace{1,\ldots,1}_j}}}(1-z)\Li_{\bf k'}(z),\label{euler}
\end{equation}
where the sum on the right-hand side runs over indices ${\bf k'}$ and integers $j\ge0$ that satisfy
$\vert{\bf k'}\vert+j\le \vert{\bf k}\vert$, and $c_{\bf k}({\bf k'};j)$ is a $\mathbb{Q}$-linear combination
of multiple zeta values of weight $\vert{\bf k}\vert-\vert{\bf k'}\vert-j$. We understand
$\Li_{\emptyset}(z)=1$ and $\vert\emptyset\vert=0$ for the empty index $\emptyset$, and the constant $1$ is interpreted as 
a multiple zeta value of weight $0$.
\end{lemma}

From this, we can obtain formulas expressing $\xi(\kk;s)$ in terms of multiple zeta functions, which can be 
regarded as a general answer to the above problem. However, we should note that there are no closed 
formulas for the coefficients $c_{\bf k}({\bf k'};j)$, and we can only compute them inductively from low weights.

\begin{theorem}[\cite{KT}\ Theorem 3.6]\label{xibyzeta}  Let ${\bf k}$ be any index set.  
The function $\xi({\bf k};s)$ can be written 
in terms of multiple zeta functions as
\begin{equation}
\xi({\bf k};s)=\sum_{{\bf k'},\,j\ge0}c_{\bf k}({\bf k'};j)\binom{s+j-1}{j}\zeta({\bf k'};s+j). \label{xi-mzv}
\end{equation}
Here, the sum is over indices ${\bf k'}$ and integers $j\ge0$ satisfying 
$\vert{\bf k'}\vert+j\le \vert{\bf k}\vert$, and $c_{\bf k}({\bf k'};j)$ is a $\mathbb{Q}$-linear combination
of multiple zeta values of weight $\vert{\bf k}\vert-\vert{\bf k'}\vert-j$.  The index ${\bf k'}$ may
be $\emptyset$ and for this we set $\zeta(\emptyset;s+j)=\zeta(s+j)$.
\end{theorem}

As an example, we used the identity 
\begin{equation}
{\rm Li}_{2,1}(1-z)=2{\rm Li}_3(z)-\log z\cdot {\rm Li}_2(z)-\zeta(2)\log z-2\zeta(3), \label{equ-6-1}
\end{equation}
obtained by integrating the well-known 
\begin{equation}
{\rm Li}_{2}(1-z)+{\rm Li}_2(z)=\zeta(2)-\log z\log(1-z). \label{equ-6-2}
\end{equation}
Applying \eqref{equ-6-1} to the definition of $\xi$ in \eqref{xidef}, we obtained
\begin{equation}
\xi(2,1;s)=2\zeta(3;s)+s\zeta(2;s+1)+\zeta(2)s\zeta(s+1)-2\zeta(3)\zeta(s). \label{equ-6-3}
\end{equation}

Lemma \ref{eulcon} (and its proof in \cite{KT}) gives an inductive way to compute the functional equation under  $z \mapsto 1-z$.
Here we give a further example which implies a multiple version of \eqref{equ-6-3}. The following identity is an example of 
Lemma~\ref{eulcon} because 
$(\log z)^n=(-1)^nn!\,{\rm Li}_{\tiny \underbrace{1,\ldots,1}_{n}}(1-z)$ (see e.g. \cite[Lemma 1]{AK1999}).

\begin{lemma}\label{L-6-1}\ For $r\in \mathbb{Z}_{\geq 0}$ and $0<z<1$, 
\begin{align}\label{equ-6-4}
& (-1)^r {\rm Li}_{2,{\tiny \underbrace{1,\ldots,1}_{r}}}(1-z)\\
& = -(r+1){\rm Li}_{r+2}(z)+(\log z){\rm Li_{r+1}}(z)\notag\\
& \quad +\sum_{j=0}^{r}\frac{r-j+1}{j!}\zeta(r-j+2)(\log  z)^j.\notag
\end{align}
\end{lemma}

\begin{proof}
We proceed by induction on $r$. When $r=0$, \eqref{equ-6-4} is nothing but \eqref{equ-6-2}. For the case $r\geq 1$, 
if we differentiate the right-hand side of \eqref{equ-6-4}, the result is equal to 
$$(-1)^{r-1} {\rm Li}_{2,{\tiny \underbrace{1,\ldots,1}_{r-1}}}(1-z)\frac{1}{z}=(-1)^r\frac{d}{dz}{\rm Li}_{2,{\tiny \underbrace{1,\ldots,1}_{r}}}(1-z),$$
by the induction hypothesis for the case of $r-1$. Hence the assertion for the case of $r$
follows from integration,  by noting the both sides of \eqref{equ-6-4} tend to 0 when $z\to 1$. 
Thus we complete the proof.
\end{proof}

Applying \eqref{equ-6-4} with $z=e^{-t}$ $(t>0)$ to \eqref{xidef}, we obtain the following generalization of \eqref{equ-6-3}.

\begin{theorem}\label{Th-6-2}\ For $r\in \mathbb{Z}_{\geq 1}$, 
\begin{align} \label{equ-6-5}
& (-1)^r\xi(2,\underbrace{1,\ldots,1}_{r};s) \\
& = -(r+1)\zeta(r+2;s)-s\zeta(r+1;s+1)\notag\\
& +\sum_{j=0}^{r}(-1)^j(r-j+1)\zeta(r-j+2)\binom{s+j-1}{j}\zeta(s+j).\notag
\end{align}
\end{theorem}

\begin{example}\label{Ex-6-3}
The case $r=1$ is \eqref{equ-6-3} and the case $r=2$ is
\begin{align*}
\xi(2,1,1;s)& =-3\zeta(4;s)+3\zeta(4)\zeta(s)-s\zeta(3;s+1)\\
& \ \ -2s\zeta(3)\zeta(s+1)+\frac{s(s+1)}{2}\zeta(2)\zeta(s+2).
\end{align*}
These coincide with the formula in \cite[Example 3.8]{KT}.
\end{example}

%%%%%%%%%%%%%%%%%%%%%%%%%%%%%%%%%%%%%%%%%%%%%%%%
\section{The function $\eta(\kk;s)$ for nonpositive indices and related topics}\label{sec-4}
%%%%%%%%%%%%%%%%%%%%%%%%%%%%%%%%%%%%%%%%%%%%%%%%

In this section, we consider multi-polylogarithms with nonpositive indices. 

\begin{lemma}[\cite{KT}\ Lemma 4.1]\label{L-4-1} \ For $k_1,\ldots,k_r\in \mathbb{Z}_{\geq 0}$, there exists a polynomial 
$P(x;\kk)\in \mathbb{Z}[x]$ such that
\begin{align}
& \Li_{\km}(z)=\frac{P(z;\kk)}{(1-z)^{k_1+\cdots+k_r+r}}, \label{4-1}\\
& {\rm deg}\,P(x;\kk)\label{4-2}\\
& =
\begin{cases}
r & (k_1=\cdots=k_r=0)\\
k_1+\cdots+k_r+r-1 & (\text{\rm otherwise}),
\end{cases}
\notag\\
& x^r\mid P(x;\kk).\label{4-3}
\end{align}
Specifically, $P(x;\underbrace{0,0,\ldots,0}_{r})=x^r$.
\end{lemma}

The case of $r=1$ is well-known (see, for example, Shimura \cite[Equations (2.17),\,(4.2) and (4.6)]{Shimura}). For example, 
$$\Li_0(z)=\frac{z}{1-z}, \quad \Li_{-1}(z)=\frac{z}{(1-z)^2}.$$
However, even if we apply this definition to \eqref{xidef} as well as in the case of positive indices, 
we cannot define the function $\xi$ with nonpositive indices. In fact, if we set, for example,
\begin{align*}
& \xi_0(s)=\frac{1}{\Gamma(s)}\int_{0}^\infty t^{s-1}\frac{\Li_0(1-e^{-t})}{e^t-1}dt=\frac{1}{\Gamma(s)}\int_{0}^\infty t^{s-1}dt,\\
& \xi_{-1}(s)=\frac{1}{\Gamma(s)}\int_{0}^\infty t^{s-1}\frac{\Li_{-1}(1-e^{-t})}{e^t-1}dt=\frac{1}{\Gamma(s)}\int_{0}^\infty t^{s-1}e^{t} dt,
\end{align*}
we see that these integrals are divergent for any $s\in \mathbb{C}$.

On the other hand, we can define the function $\eta$ with nonpositive indices as follows.

\begin{defn}\label{Def-Main-2} 
For $k_1,\ldots,k_r\in \mathbb{Z}_{\geq 0}$, define 
\begin{equation}
\eta(\km;s)=\frac{1}{\Gamma(s)}\int_{0}^\infty t^{s-1}\frac{\Li_{\km}(1-e^t)}{1-e^t}dt \label{4-5}
\end{equation}
for $s\in \mathbb{C}$ with ${\rm Re}(s)>1-r$. In the case $r=1$, denote $\eta(-k;s)$ by $\eta_{-k}(s)$.
\end{defn}

We can easily check that the integral on the right-hand side of \eqref{4-5} is absolutely convergent for 
${\rm Re}(s)>1-r$. 
Similar to the case with positive indices, we can see that $\eta(\km;s)$ can be analytically continued to 
an entire function on the whole complex plane, and satisfies 
\begin{equation}
\eta(\km;-m)=\bb_{m}^{(\km)}\quad (m\in \mathbb{Z}_{\geq 0}) \label{4-6}
\end{equation}
for $k_1,\ldots,k_r\in \mathbb{Z}_{\geq 0}$. 
In particular when $r=1$, we have
\begin{equation}
\eta_{-k}(-m)=\bb_{m}^{(-k)}\quad (k\in \mathbb{Z}_{\geq 0},\ m\in \mathbb{Z}_{\geq 0}). \label{eta-negative}
\end{equation}

Furthermore, we modify the definition \eqref{xidef} as follows.

\begin{defn}\label{Def-xi-tilde} 
For $k_1,\ldots,k_r\in \mathbb{Z}_{\geq 0}$ with $(k_1,\ldots,k_r)\neq (0,\ldots,0)$, define 
\begin{equation}
\widetilde{\xi}(\km;s)=\frac{1}{\Gamma(s)}\int_{0}^\infty t^{s-1}\frac{\Li_{\km}(1-e^t)}{e^{-t}-1}dt \label{4-5-2}
\end{equation}
for $s\in \mathbb{C}$ with ${\rm Re}(s)>1-r$. In the case $r=1$, denote $\widetilde{\xi}(-k;s)$ by 
$\widetilde{\xi}_{-k}(s)$ for $k\geq 1$.
\end{defn}

We see that $\widetilde{\xi}(\km;s)$ can be analytically continued to an entire function on the whole complex plane, 
and satisfies 
\begin{equation}
\widetilde{\xi}(\km;-m)=\cc_{m}^{(\km)}\quad (m\in \mathbb{Z}_{\geq 0}) \label{4-6-2}
\end{equation}
for $k_1,\ldots,k_r\in \mathbb{Z}_{\geq 0}$ with $(k_1,\ldots,k_r)\neq (0,\ldots,0)$. 
In particular, $\widetilde{\xi}_{-k}(-m)=\cc_{m}^{(-k)}$ $(k\in \mathbb{Z}_{\geq 1},\ m\in \mathbb{Z}_{\geq 0})$. 

\begin{remark}
Note that we cannot define $\widetilde{\xi}(\kk;s)$ by replacing $(\km)$ with $(\kk)$ in \eqref{4-5-2}. In fact, if we set, for example,
\begin{equation*}
\widetilde{\xi}_1(s)=\frac{1}{\Gamma(s)}\int_{0}^\infty t^{s-1}\frac{\Li_{1}(1-e^t)}{e^{-t}-1}dt=s\zeta(s+1)+\frac{1}{\Gamma(s)}\int_{0}^\infty t^{s}dt,
\end{equation*}
which is not convergent for any $s\in \mathbb{C}$.
\end{remark} 

Here we extend definitions of poly-Bernoulli numbers \eqref{1-1} and \eqref{1-2} as follows. For $s\in \mathbb{C}$, 
we define
\begin{align}
&\frac{{\rm Li}_{s}(1-e^{-t})}{1-e^{-t}}=\sum_{n=0}^\infty \bb_n^{(s)}\frac{t^n}{n!},  \label{ex-1-1}\\
&\frac{{\rm Li}_{s}(1-e^{-t})}{e^t-1}=\sum_{n=0}^\infty \cc_n^{(s)}\frac{t^n}{n!},  \label{ex-1-2}
\end{align}
where 
\begin{equation}
{\rm Li}_{s}(z)=\sum_{m=1}^\infty \frac{z^m}{m^s}\quad (|z|<1). \label{e-1-3}
\end{equation}
Using 
\begin{equation}
%\prod_{j=1}^r \frac{e^{\sum_{\nu=j}^{r}x_\nu}(1-e^t)}{1-e^{\sum_{\nu=j}^{r}x_\nu}(1-e^t)}
%=\sum_{\kk \geq 0}\Li_{\km}(1-e^t)\frac{x_1^{k_1}\cdots x_r^{k_r}}{k_1!\cdots k_r!}, \label{4-10}
\frac{e^{x}(1-e^t)}{1-e^{x}(1-e^t)}
=\sum_{k=0}^\infty\Li_{-k}(1-e^t)\frac{x^{k}}{k!}, \label{4-10}
\end{equation}
we have the following.

\begin{theorem}[\cite{KT}\ Theorem 4.7]\label{Th-4-6}\ \ For $k \in \mathbb{Z}_{\geq 0}$, 
\begin{equation}
\eta(-k;s)=B_{k}^{(s)}.   \label{4-8}
\end{equation}
\end{theorem}

Setting $s=-n\in \mathbb{Z}_{\leq 0}$ in \eqref{4-8} % in the case $r=1$ 
and using \eqref{eta-negative}, 
we obtain the duality relation $B_{n}^{(-k)}=B_{k}^{(-n)}$ in \eqref{1-4}, which can be written as 
\begin{equation}
\eta_{-k}(-n)=\eta_{-n}(-k). \label{eta-dual}
\end{equation}
Similarly, we can prove that
\begin{equation}
\widetilde{\xi}_{-k-1}(-n)=\widetilde{\xi}_{-n-1}(-k)\quad (n,k\in \mathbb{Z}_{\geq 0}), \label{tildexi-dual0}
\end{equation}
namely the duality relation $C_{n}^{(-k-1)}=C_{k}^{(-n-1)}$ in \eqref{1-5}. 

On the other hand, for $n,k\in \mathbb{Z}_{\geq 1}$, we found experimentally the identities (\cite[Eq. (36)]{KT})
\begin{equation}
\eta_{k}(n)=\eta_{n}(k), \label{eta-dual-2}
\end{equation}
which was soon proved and generalized by Yamamoto \cite{Yamamoto}. In particular when $r=1$, he showed 
\begin{equation}
\eta_{u}(s)=\eta_{s}(u) \label{eta-dual-3}
\end{equation}
for $s,u\in \mathbb{C}$, where
\begin{equation}
\eta_u(s)=\frac{1}{\Gamma(s)}\int_{0}^\infty t^{s-1}\frac{\Li_{u}(1-e^t)}{1-e^t}dt\  (s,u\in \mathbb{C};\ \Re(s)>1), \label{Yamamoto-eta}
\end{equation}
which can be analytically continued to $(s,u)\in \mathbb{C}^2$. More recently Kawasaki and Ohno gave 
an alternative proof of \eqref{eta-dual-2} in \cite{KO}.

Inspired by Yamamoto's result, Komori and the second named author \cite{Ko-Tsu} consider a more general type 
of zeta function denoted by $\xi_D(u,s;y,w;g)$ $(u,s,y,w\in \mathbb{C};g\in GL(2,\mathbb{C}))$ which satisfies
\begin{equation}
\xi_D(u,s;y,w-1;g)=-\frac{1}{{\rm det}\,g}\xi_D(s,u;w,y-1;g^{-1}) \label{Ko-Tsu-dual}
\end{equation}
(see \cite[Theorem 4.3]{Ko-Tsu}). In particular, 
for $g_\eta=\begin{pmatrix} -1 & 1 \\ 0 & 1 \end{pmatrix}$, we have $\xi_D(u,s;1,0;g_\eta)=\eta_u(s)$. 
Hence \eqref{Ko-Tsu-dual} in this case implies \eqref{eta-dual-3}. It is also shown that
\begin{align}
& \xi_D(u-1,s;y,w-2;g_\eta)+(1-y)\xi_D(u,s;y,w-2;g_\eta)\label{Ko-Tsu-dual-2}\\
& =\xi_D(s-1,u;w,y-2;g_\eta)+(1-w)\xi_D(s,u;w,y-2;g_\eta)\notag 
\end{align}
(see \cite[Theorem 4.3]{Ko-Tsu}). Here we note that 
$\xi_D(u,s;1,-1;g)=\widetilde{\xi}_u(s)$ which is defined by replacing $-k$ with $u$ in the definition of 
$\widetilde{\xi}_{-k}(s)$ (see Definition \ref{Def-xi-tilde}). Hence \eqref{Ko-Tsu-dual-2} with $(y,w)=(1,1)$ implies
\begin{equation}
\widetilde{\xi}_{u-1}(s)=\widetilde{\xi}_{s-1}(u), \label{tildexi-dual}
\end{equation}
which includes \eqref{tildexi-dual0}. 

Furthermore, Yamamoto proved the identity (\cite[\S1]{Yamamoto})
\begin{equation*}
\eta_k(n)=\sum_{0<a_1\leq \cdots \leq a_k=b_n\geq \cdots \geq b_1>0}\frac{1}{a_1\cdots a_kb_1\cdots b_n}\quad (k,n\in \mathbb{Z}_{\geq 1}),
\end{equation*}
which directly reveals the symmetry \eqref{eta-dual-2}. Similar expression for $\xi_k(n)$ is 
\begin{equation*}
\xi_k(n)=\sum_{0<a_1= \cdots = a_k=b_n\geq \cdots \geq b_1>0}\frac{1}{a_1\cdots a_kb_1\cdots b_n}\quad (k,n\in \mathbb{Z}_{\geq 1}),
\end{equation*}
which unfortunately is not symmetric.  We do not know if any duality property holds for $\xi_k(s)$. 

In addition, recall that we mention at the end of \S3 in \cite{KT} the identity
\begin{align}
&\eta_k(m) =\binom{m+k}{k}\zeta(m+k) \label{eta-val-KT}\\
&  -\!\!\!\!\sum_{2\le r\le k+1\atop 
j_1+\cdots+j_r=m+k-r-1}\!\!\!\binom{j_1+\cdots+j_{r-1}}{k-r+1}\cdot \zeta(j_1+1,\cdots,j_{r-1}+1,j_r+2),\notag
\end{align}
without proof. Recently Shingu proved
\begin{align}
\eta_k(m) &=\sum_{k_1+\cdots+k_r=k+m \atop 1\leq r \leq k,\ k_r\geq 2}\sum_{i=1}^{k_r-1}\binom{k+m-r-i}{m-i}\zeta(k_1,\ldots,k_r) \label{eta-val-Shingu}
\end{align}
in his master's thesis \cite{Shingu} by using Yamamoto's multiple integrals introduced in \cite{Yamamoto-B}. 
It is easy to derive \eqref{eta-val-KT} from \eqref{eta-val-Shingu}.

The referee pointed out that \eqref{eta-val-Shingu} should be equivalent to \eqref{3-8-2}
via the standard relation 
\begin{equation}\label{ord-star} 
\zeta^\star({\bf k})=\sum_{{\bf k'}\preceq {\bf k}}\zeta({\bf k'}).
\end{equation}
We have checked that \eqref{3-8-2} actually implied \eqref{eta-val-Shingu}, by computing
how many times each  $\zeta(k_1,\ldots,k_r)$ appeared when we wrote each 
$\zeta^\star(j_1,\ldots,j_{k-1},j_k)$ in \eqref{3-8-2} as a sum of ordinary multiple zeta
values using \eqref{ord-star}. Our computation is not too complicated but a little lengthy
using generating series, and we omit the details here. (We have not checked the opposite
implication, but it should be done in a similar vein.)\\

At the end of this section, we consider an application of the duality relation $\eta(k;n)=\eta(n;k)$ in \eqref{eta-dual-2}. 
By combining Proposition \ref{etaxi} and Theorem \ref{xibyzeta}, we obtain, for $k,n\in \mathbb{Z}_{\geq 1}$,  
\begin{align*}
\eta(k;n)&= \sum_{(k)\preceq{\bf k'}}\xi({\bf k'};n)\\
& =\sum_{(k)\preceq{\bf k'}}\ \sum_{{\bf k''},\,j\ge0}c_{\bf k'}({\bf k''};j)\binom{n+j-1}{j}\zeta({\bf k''};n+j),
\end{align*}
where the sum is over indices ${\bf k''}$ and integers $j\ge0$ satisfying 
$\vert{\bf k''}\vert+j\le \vert{\bf k'}\vert$, and $c_{\bf k'}({\bf k''};j)$ is a $\mathbb{Q}$-linear combination
of multiple zeta values of weight $\vert{\bf k'}\vert-\vert{\bf k''}\vert-j$ determined by \eqref{xi-mzv}.  

We see that Proposition \ref{etaxi} and Theorem \ref{xibyzeta} were given by 
the connection formulas of Euler type and Landen type, respectively. From \eqref{eta-dual-2}, we obtain the following.

\begin{theorem}\label{Th-6-1}\ 
With the above notation, for $k,n\in \mathbb{Z}_{\geq 1}$,  
\begin{align}
& \sum_{(k)\preceq{\bf k'}}\ \sum_{{\bf k''},\,j\ge0}c_{\bf k'}({\bf k''};j)\binom{n+j-1}{j}\zeta({\bf k''};n+j) \label{eq-6-1}\\
& =\sum_{(n)\preceq{\bf n'}}\ \sum_{{\bf n''},\,j\ge0}c_{\bf n'}({\bf n''};j)\binom{k+j-1}{j}\zeta({\bf n''};k+j).\notag
\end{align}
\end{theorem}

\begin{example}\label{Ex-6-4}
For example, set $(k,n)=(3,2)$ in \eqref{eq-6-1}. Then, by \cite[Example 3.8]{KT}, we have
\begin{align*}
& \zeta(1,2,2)+\zeta(2,1,2)+2\zeta(1,1,3)-\zeta(2)\zeta(1,2)+\zeta(3,2)-3\zeta(1,4)\\
& +2\zeta(2)\zeta(3)+4\zeta(5)=6\zeta(5)-3\zeta(1,4)-\zeta(2,3)+\zeta(2)\zeta(3).
\end{align*}
This can of course be checked by known identities, for example, double shuffle relations.
We do not pursue here connections between identities of MZVs obtained by $\eta(k;n)=\eta(n;k)$ 
as above and known sets of identities. Are there some interesting aspects?
\end{example}

\section{Zeta functions interpolating multiple zeta values of level $2$}\label{sec-6}
%%%%%%%%%%%%%%%%%%%%%%%%%%%%%%%%%%%%%

In this section, we define a certain level $2$-version of the function $\xi(k_1,\ldots,k_r;s)$ 
which interpolates multiple zeta values of level $2$ at positive integers.  Here, we mean by
MZVs of level 2 the quantities essentially equivalent to those often referred to as the Euler sums.
But we only look at a special subclass of them.  Specifically, we look at the quantity
\[ \sum_{0<m_1<\cdots <m_r\atop m_i\equiv i\bmod 2}
\frac1{m_1^{k_1}\cdots m_{r}^{k_{r}}}, \]
i.e., the sum is restricted to $m_1,m_2,m_3,\ldots$ with odd, even, odd, \ldots in alternating manner.
 These numbers in depth 2 were considered in \cite{KanekoTasaka} in connection to modular forms
of level 2, establishing a generalization of the work by Gangle-Kaneko-Zagier \cite{GKZ}. 

In \cite[Section 4]{Sasaki2012}, Sasaki considered the polylogarithm of level $2$ defined by
$${\rm Ath}_k(z)=\sum_{n=0}^\infty \frac{z^{2n+1}}{(2n+1)^k}=\Li_k(z)-\frac{1}{2^k}\Li_k(z^2)$$
for $k\in \mathbb{Z}$. When $k=1$, this becomes the well-known 
$${\rm Ath}_1(z)=\tanh^{-1}z=\sum_{n=0}^\infty \frac{z^{2n+1}}{2n+1}=\Li_1(z)-\frac{1}{2}\Li_1(z^2).$$

We generalize this to a multiple version. For $k_1,\ldots,k_r\in \mathbb{Z}$, define
\begin{align}
& {\rm Ath}(k_1,\ldots,k_r;z)=\sum_{0<m_1<\cdots <m_r\atop m_i\equiv i\bmod 2}
\frac{z^{m_r}}{m_1^{k_1}\cdots m_{r}^{k_{r}}}\\ \label{def-MP-2}
& \qquad =\sum_{n_1,\ldots,n_r=0}^\infty \frac{z^{\sum_{\nu=1}^{r}(2n_\nu+1)}}{\prod_{j=1}^{r}\left(\sum_{\nu=1}^{j}(2n_\nu+1)\right)^{k_j}}. \notag
\end{align}
Note that since ${\rm Ath}(1;z)=\tanh^{-1}z$, we have
\begin{equation}
{\rm Ath}(1;\tanh t)=t. \label{ee-6-0}
\end{equation}
Similar to \cite[Lemma 1]{AK1999}, we can easily obtain the following.

\begin{lemma}\label{Lem-6-1} 
\begin{enumerate}[{\rm (i)}]
\item For $k_1,\ldots,k_r\in \mathbb{Z}_{\geq 1}$, 
\begin{align*}
&\frac{d}{dz}{\rm Ath}(k_1,\ldots,k_r;z)\\
& =
\begin{cases}
\frac{1}{z}{\rm Ath}(k_1,\ldots,k_{r-1},k_r-1;z) & (k_r\geq 2),\\
\frac{1}{1-z^2}{\rm Ath}(k_1,\ldots,k_{r-1};z) & (k_r=1).
\end{cases}
\end{align*}
\item $\displaystyle{{\rm Ath}(\underbrace{1,\ldots,1}_{r};z)=\frac{1}{r!}({\rm Ath}(1;z))^r}.$
\end{enumerate}
\end{lemma}

We define a kind of multiple zeta function of level $2$ as follows.

\begin{defn}\label{Def-6-1}\ For $k_1,\ldots,k_{r-1}\in \mathbb{Z}_{\geq 1}$ and $\Re s>1$, 
let
\begin{align}
&\T(k_1,\ldots,k_{r-1},s)=\sum_{0<m_1<\cdots <m_r\atop m_i\equiv i\bmod 2}
\frac1{m_1^{k_1}\cdots m_{r-1}^{k_{r-1}}m_r^s} \label{ee-5-2} \\ 
&\qquad=\sum_{n_1,\ldots,n_r\geq 0}\prod_{j=1}^{r-1}\left(\sum_{\nu=1}^{j}(2n_\nu+1)\right)^{-k_j}
\times \left(\sum_{\nu=1}^{r}(2n_\nu+1)\right)^{-s}.\notag
\end{align}
Furthermore, as its normalized version, let
\begin{align}
\TT(k_1,\ldots,k_{r-1},s)&=2^{r}\T(k_1,\ldots,k_{r-1},s).\label{ee-5-2-2}
\end{align}
\end{defn}

When $k_r>1$, we see that 
\[ {\rm Ath}(k_1,\ldots,k_r;1) = \T(k_1\ldots,k_r). \]
Corresponding to these functions, we define a level $2$-version of $\xi(k_1,\ldots,k_r;s)$.

\begin{defn}\label{Def-6-2}
For $k_1,\ldots,k_r\in \mathbb{Z}_{\geq 1}$, let
\begin{align}
& \psi(k_1,\ldots,k_r;s)\label{ee-6-1}\\
& =\frac{2^r}{\Gamma(s)}\int_0^\infty t^{s-1}\frac{{\rm Ath}(k_1,\ldots,k_{r};\tanh (t/2))}{\sinh(t)}\,dt \quad (\Re s>0). \notag
\end{align}
\end{defn}

\begin{remark}\label{Rem-6-3}
In \cite[Section 4]{Sasaki2012}, Sasaki essentially considered \eqref{ee-5-2}, and also $\psi(k_1;s)$ . 
In fact, Sasaki considered a little more general function $\psi_k(s,a)$ $(0<a<1)$, and our $\psi(k;s)$ coincides with 
his $2^{s+2}\psi_k(s,1/2)$. 
\end{remark}

Similar to \cite[Theorem 6]{AK1999}, we can see that $\psi(k_1,\ldots,k_r;s)$ can be continued to $\mathbb{C}$ as an entire function. 
Further we can prove the following theorem which is exactly a level $2$-analogue of \cite[Theorem 8]{AK1999}. 
Note that this theorem for the case $r=1$ was essentially proved by Sasaki (see \cite[Theorem 7]{Sasaki2012}). 

\begin{theorem}\label{Th-6-4} For $r, k\in \mathbb{Z}_{\geq 1}$, 
\begin{align*}
& \psi(\underbrace{1,\ldots,1}_{r-1},k;s)   \\
&=(-1)^{k-1}\!\!\!\!\!\!\sum_{a_1,\ldots,a_k\geq 0 \atop a_1+\cdots+a_k=r}
\binom{s+a_k-1}{a_k}\cdot \TT(a_1+1,\ldots,a_{k-1}+1,a_k+s)\notag\\
&\quad +\sum_{j=0}^{k-2}(-1)^{j}\,\TT(\underbrace{1,\ldots,1}_{r-1},k-j)\cdot 
\TT(\underbrace{1,\ldots,1}_{j},s).\notag
\end{align*}
\end{theorem}

In order to prove this theorem, we prepare the following lemma which is a level $2$-version of \cite[Theorem 3 (i)]{AK1999}. 
The proof is completely similar and is omitted.

\begin{lemma}\label{Lem-6-5}\ For $l_1,\ldots,l_{m-1}\in \mathbb{Z}_{\geq 1}$ and $\Re s>1$,  
\begin{align*}
&\TT(l_1,\ldots,l_{m-1},s) \\
&= \frac{1}{\Gamma(l_1)\cdots\Gamma(l_{m-1})\Gamma(s)}\int_{0}^\infty \cdots \int_{0}^\infty x_1^{l_1-1}\cdots x_{m-1}^{l_{m-1}-1}x_m^{s-1}\\
& \qquad \times \prod_{j=1}^{m}\frac{1}{\sinh(x_j+\cdots+x_m)}dx_1\cdots dx_m.
\end{align*}
\end{lemma}

\begin{proof}[Proof of Theorem \ref{Th-6-4}] 
The method of the proof is similar to that in \cite[Theorem 8]{AK1999} (see also \cite[Theorem 7]{Sasaki2012}). 
Given $r,k\ge1$, introduce the following integrals 
\begin{align*}
I_\nu^{(r,k)}(s)& =\frac{2^r}{\Gamma(s)}\! \int_{0}^\infty \!\!\!\cdots\! \int_{0}^\infty 
\frac{{\rm Ath}(\overbrace{1,\ldots,1}^{r-1},\nu;\tanh ((x_\nu+\cdots+x_k)/2))}{\prod_{l=\nu}^{k}
\sinh(x_l+\cdots+x_k)}\\
&\qquad \times x_k^{s-1}\,dx_\nu\cdots dx_k.
\end{align*}
We compute $I_1^{(r,k)}(s)$ in two different ways.  First, since
\[ {\rm Ath}(\underbrace{1,\ldots,1}_{r};\tanh ((x_1+\cdots+x_k)/2))=
\frac1{r!}\left(\frac{x_1+\cdots+x_k}{2}\right)^r  \]
by Lemma~\ref{Lem-6-1}~(ii) and \eqref{ee-6-0}, we have
\begin{align*}
&I_1^{(r,k)}(s)\\
& =\frac{1}{\Gamma(s)\, r!}\int_{0}^\infty \cdots \int_{0}^\infty 
\frac{(x_1+\cdots+x_k)^rx_k^{s-1}}{\prod_{l=1}^{k}\sinh(x_l+\cdots+x_k)}\,dx_1\cdots dx_k\\
& =\frac{1}{\Gamma(s)}\!\!\sum_{a_1+\cdots+a_k=r}
\frac{1}{a_1!\cdots a_k!}\int_{0}^\infty\!\!\! \cdots\! \int_{0}^\infty x_1^{a_1}\cdots x_{k-1}^{a_{k-1}}x_k^{s+a_k-1}\\
& \qquad \times \frac1{\prod_{l=1}^{k}\sinh(x_l+\cdots+x_k)}\,dx_1\cdots dx_k\\
&=\sum_{a_1+\cdots+a_k=r}\frac{\Gamma(s+a_k)}{\Gamma(s)a_k!}
\times \frac1{\Gamma(a_1+1)\cdots\Gamma(a_{k-1}+1)\Gamma(s+a_k)}\\
&\qquad\times \int_{0}^\infty\!\!\! \cdots\! \int_{0}^\infty 
\frac{x_1^{a_1}\cdots x_{k-1}^{a_{k-1}}x_k^{s+a_k-1}}
{\prod_{l=1}^{k}\sinh(x_l+\cdots+x_k)}\,dx_1\cdots dx_k.
\end{align*}
Using Lemma~\ref{Lem-6-5} for the last integral, we obtain
\begin{align}
 \label{I1-1}   I_1^{(r,k)}(s) &=   \sum_{a_1+\cdots+a_k=r}\binom{s+a_k-1}{a_k}\times\\
&\qquad \qquad \TT(a_1+1,\ldots,a_{k-1}+1,s+a_k).\notag
\end{align}

Secondly, by using 
\begin{align} &\frac{\partial}{\partial x_\nu} {\rm Ath}(\underbrace{1,\ldots,1}_{r-1},\nu+1;
\tanh ((x_\nu+\cdots +x_k)/2))\label{deriv} \\
&=\frac{{\rm Ath}(\overbrace{1,\ldots,1}^{r-1},\nu;\tanh ((x_\nu+\cdots +x_k)/2))}
{\sinh(x_\nu+\cdots+x_k)}\notag 
\end{align} (see Lemma~\ref{Lem-6-1}) and Lemma~\ref{Lem-6-5}, we compute
\begin{align*}
&I_\nu^{(r,k)}(s)\\
&=\frac{2^r}{\Gamma(s)}  \int_{0}^\infty\!\!\! \cdots\! \int_{0}^\infty 
\left[  {\rm Ath}(\underbrace{1,\ldots,1}_{r-1},\nu+1;\tanh ((x_\nu+\cdots +x_k)/2))
\right]_{x_\nu=0}^\infty\\
&\qquad \times \frac1{\prod_{l=\nu+1}^k \sinh(x_l+\cdots+x_k)}\, x_k^{s-1}\,dx_{\nu+1}\cdots dx_k\\
&=2^r\T(\underbrace{1,\ldots,1}_{r-1},\nu+1)\cdot
\TT(\underbrace{1,\ldots,1}_{k-\nu-1},s) -I_{\nu+1}^{(r,k)}\\
&=\TT(\underbrace{1,\ldots,1}_{r-1},\nu+1)\cdot
\TT(\underbrace{1,\ldots,1}_{k-\nu-1},s) -I_{\nu+1}^{(r,k)}.
\end{align*}
Therefore, using this relation repeatedly, we obtain 
\begin{align*}
&I_1^{(r,k)}(s)\label{I1-2}\\
&=\sum_{\nu=1}^{k-1}(-1)^{\nu-1}\TT(\underbrace{1,\ldots,1}_{r-1},\nu+1)\cdot
\TT(\underbrace{1,\ldots,1}_{k-\nu-1},s)+(-1)^{k-1}I_{k}^{(r,k)}\notag\\
&=\sum_{j=0}^{k-2}(-1)^{k-j}\TT(\underbrace{1,\ldots,1}_{r-1},k-j)\cdot
\TT(\underbrace{1,\ldots,1}_{j},s)+(-1)^{k-1}I_{k}^{(r,k)}.\notag
\end{align*}
By definition, we have 
\[ I_k^{(r,k)}(s)=\psi(\underbrace{1,\ldots,1}_{r-1},k;s), \]
and thus
\begin{align}
\label{I1-2} I_1^{(r,k)}(s)&=\sum_{j=0}^{k-2}(-1)^{k-j}\TT(\underbrace{1,\ldots,1}_{r-1},k-j)\cdot
\TT(\underbrace{1,\ldots,1}_{j},s)\\
&\qquad +(-1)^{k-1}\psi(\underbrace{1,\ldots,1}_{r-1},k;s).\notag
\end{align}
Comparing \eqref{I1-1} and \eqref{I1-2}, we obtain the assertion.
\end{proof}

Next, we show a level $2$-version of \cite[Theorem 9\ (i)]{AK1999}.

\begin{theorem}\label{T-5-1}\ For $r,k\in \mathbb{Z}_{\geq 1}$ and $m\in \mathbb{Z}_{\geq 0}$,
\begin{align}
 &\psi(\underbrace{1,\ldots,1}_{r-1},k;m+1)\label{ee-5-4}\\
 &=\sum_{a_1,\ldots,a_k\geq 0 \atop a_1+\cdots+a_k=m}\binom{a_k+r}{r}
 \cdot \TT(a_1+1,\ldots,a_{k-1}+1,a_k+r+1). \notag
\end{align}
\end{theorem}

\begin{proof}
By \eqref{deriv}, we have
\begin{align*}
& \psi({1,\ldots,1},k;m+1)\\
&=\frac{2^r}{m!}\int_0^\infty \frac{t_k^{m}}{\sinh t_k} \int_0^{t_k}
\frac{{\rm Ath}(\overbrace{1,\ldots,1}^{r-1},k-1;\tanh(t_{k-1}/2))}{\sinh t_{k-1}}
\, dt_{k-1}dt_k\\
&=\frac{2^r}{m!}\int_0^\infty \frac{t_k^{m}}{\sinh t_k} \int_0^{t_k}
\frac{1}{\sinh t_{k-1}}\int_0^{t_{k-1}}\\
&\qquad\qquad\frac{{\rm Ath}(\overbrace{1,\ldots,1}^{r-1},k-2;\tanh (t_{k-2}/2))}{\sinh t_{k-2}}
\, dt_{k-2}dt_{k-1}dt_k\\
&=\cdots\\
&=\frac{2^r}{m!}\int_0^\infty \int_0^{t_k}\cdots \int_0^{t_2}
\frac{t_k^m\, {\rm Ath}(\overbrace{1,\ldots,1}^{r};\tanh (t_1/2))}
{\sinh(t_k)\cdots\sinh(t_1)}\, dt_1\cdots dt_k\\
&=\frac{1}{m!r!}\int_0^\infty \int_0^{t_k}\cdots \int_0^{t_2}
\frac{t_k^m\,t_1^r}
{\sinh(t_k)\cdots\sinh(t_1)}\, dt_1\cdots dt_k.
\end{align*}
By the change of variables 
\[t_1=x_k, t_2=x_{k-1}+x_k,\ldots, t_k=x_1+\cdots+x_k,\]
we obtain 
\begin{align*}
&\psi(1,\ldots,1,k;m+1)\\
&=\frac{1}{m!r!}\int_0^\infty \int_0^\infty
\frac{(x_1+\cdots+x_k)^m\,x_k^r}
{\prod_{l=1}^k\sinh(x_l+\cdots+x_k)}\, dt_1\cdots dt_k\\
&=\sum_{a_1+\cdots+a_k=m}\binom{a_k+r}{r}
 \cdot \TT(a_1+1,\ldots,a_{k-1}+1,a_k+r+1). 
\end{align*}
\end{proof}

\begin{corollary}  For $r,k\ge1$, we have the ``height one" duality
\begin{equation}\label{ht1dual}
\TT(\underbrace{1,\ldots,1}_{r-1},k+1)=\TT(\underbrace{1,\ldots,1}_{k-1},r+1). 
\end{equation}
\end{corollary}

\begin{proof}
If we set $m=0$ in \eqref{ee-5-4}, we have
\begin{equation}\label{dual1} \psi(\underbrace{1,\ldots,1}_{r-1},k;1)= \TT(\underbrace{1,\ldots,1}_{k-1},r+1). \end{equation}
On the other hand, from the definition we have in general 
\begin{align*}
\psi(k_1,\ldots,k_r;1)&=2^r \int_0^\infty \frac{ {\rm Ath}(k_1,\ldots,k_r;\tanh (t/2))}{\sinh t}\,dt\\
&=2^r\int_0^\infty \frac{d}{dt} {\rm Ath}(k_1,\ldots,k_{r-1},k_r+1;\tanh (t/2))\,dt\\
&=\TT(k_1,\ldots,k_{r-1},k_r+1)
\end{align*}
and in particular
\begin{equation}\label{dual2}\psi(\underbrace{1,\ldots,1}_{r-1},k;1)= \TT(\underbrace{1,\ldots,1}_{r-1},k+1). \end{equation}
Thus from \eqref{dual1} and \eqref{dual2} we obtain \eqref{ht1dual}.
\end{proof}
We remark that, by computing $\xi(\underbrace{1,...,1}_{r-1},k;1)$ in two ways as above, we obtain an 
alternative proof of the usual height one duality 
$\zeta(\underbrace{1,...,1}_{r-1},k+1)=\zeta(\underbrace{1,...,1}_{k-1},r+1)$.

In the forthcoming paper \cite{KT2}, we extend the duality \eqref{ht1dual} in full generality.

By setting $s=m+1$ in Theorem~\ref{Th-6-4} and comparing with Theorem~\ref{T-5-1}, 
we obtain a level $2$-version of \cite[Corollary 11]{AK1999} as follows.

\begin{theorem}\label{Th-5-2}\ For $m,r\geq 1$ and $k\ge2$,
\begin{align*}
&\sum_{a_1,\ldots,a_k\geq 0 \atop a_1+\cdots+a_k=m}\binom{a_k+r}{r}\cdot \TT(a_1+1,\ldots,a_{k-1}+1,a_k+r+1) \notag\\
& \ +(-1)^{k}\sum_{a_1,\ldots,a_k\geq 0 \atop a_1+\cdots+a_k=r}\binom{a_k+m}{m}\cdot \TT(a_1+1,\ldots,a_{k-1}+1,a_k+m+1)\\
& =\sum_{j=0}^{k-2}(-1)^j \TT(\underbrace{1,\ldots,1}_{r-1},k-j)\cdot \TT(\underbrace{1,\ldots,1}_{j},m+1).
\end{align*}
\end{theorem}
If we use the duality $\TT(\underbrace{1,\ldots,1}_{j},m+1)=\TT(\underbrace{1,\ldots,1}_{m-1},j+2)$,
the right-hand side becomes the exact analogue of the one in \cite[Corollary 11]{AK1999}.

\begin{example}\label{E-5-3}\ 
We recall
\begin{align*}
&\zeta^{o}(s)\, (=\T(s))=\sum_{n=0}^\infty \frac{1}{(2n+1)^{s}}=\left(1-2^{-s}\right)\zeta(s),\\
&\zeta^{oe}(k,s)\, (=\T(k,s))=\sum_{m=0}^\infty\sum_{n=1}^\infty \frac{1}{(2m+1)^{k}(2m+2n)^{s}}
\end{align*}
(see Kaneko-Tasaka \cite{KanekoTasaka}). 
Since
\begin{align*}
& \TT(s)=2\T(s)=2\zeta^{o}(s),\\
& \TT(k,s)=2^{2}\T(k,s)=2^{2}\zeta^{oe}(k,s),
\end{align*}
Theorem \ref{Th-5-2} for the case $k=2$ and $r=1$ 
gives 
\begin{align*}
& \sum_{a=0}^{m}(a+1)\zeta^{oe}(m-a+1,a+2)\\
& +\zeta^{oe}(2,m+1)+(m+1)\zeta^{oe}(1,m+2)=\zeta^{o}(2)\zeta^{o}(m+1).
\end{align*}
\end{example}

\begin{remark}\label{R-6-1} 
We have introduced the function $\psi(k_1,\ldots,k_r;s)$ as a level $2$-version of $\xi(k_1,\ldots,k_r;s)$,
and proved results corresponding to those in \cite{AK1999}.  
In the forthcoming paper \cite{KT2}, we will further discuss level $2$-versions of poly-Bernoulli numbers
and multiple zeta values in connection to $\psi(k_1,\ldots,k_r;s)$, and hopefully, a version 
corresponding to $\eta(k_1,\ldots,k_r;s)$
\end{remark}

\ 

\if0
{\bf Acknowledgements.}\ 
The authors would like to express their sincere gratitude to the referee for valuable suggestions and comments,
in particular on the equivalence of \eqref{eta-val-Shingu} and \eqref{3-8-2} via \eqref{ord-star}. 
This work was supported by JSPS KAKENHI Grant numbers 16H06336 (M. Kaneko)
and 18K03218 (H. Tsumura).
\fi

\end{document}